\documentclass[12pt,reqno]{amsart}
\usepackage{amsthm}

\usepackage{amssymb}
\usepackage{graphics}
\usepackage{tikz}
\usetikzlibrary{shapes,backgrounds,calc}
\usepackage{latexsym}
\usepackage{multicol}
\usepackage{verbatim,enumerate}
\usepackage{accents}
\usepackage{cite}

\usepackage[colorlinks=true, linkcolor=blue, citecolor=blue, urlcolor=blue]{hyperref}
\usepackage{hyperref}
\usepackage{amsmath, amscd,url}

\usepackage{setspace}
\usepackage{pstricks}

\advance\textwidth by 1.3in \advance\oddsidemargin by -.6in \advance\evensidemargin by -.6in
\theoremstyle{definition}
\newtheorem{cor}{Corollary}
\newtheorem{lem}{Lemma}

\newtheorem{thm}{Theorem}

\theoremstyle{definition}
\newtheorem{defn}{Definition}

\theoremstyle{definition}
\newtheorem{example}{Example}
\newtheorem{rem}{Remark}

\newenvironment{pf}{\proof}{\endproof}
\newcounter{cnt}
 \makeatletter
\def\mydggeometry{\makeatletter\dg@YGRID=1\dg@XGRID=20\unitlength=0.003pt\makeatother}
\makeatother \theoremstyle{remark}

\numberwithin{equation}{section}
\let\bwdg\bigwedge
\def\bigwedge{{\textstyle\bwdg}}

\newcommand{\nc}{\newcommand}
\newcommand{\rnc}{\renewcommand}

\nc{\cal}{\mathcal} \nc{\goth}{\mathfrak} \rnc{\bold}{\mathbf}

\nc\bomega{{\mbox{\boldmath $\omega$}}} \nc\bpsi{{\mbox{\boldmath $\Psi$}}}
 \nc\balpha{{\mbox{\boldmath $\alpha$}}}
 \nc\bpi{{\mbox{\boldmath $\pi$}}}
 \nc\bvpi{{\mbox{\boldmath $\varpi$}}}
\nc\chara{\operatorname{ch}}

  \nc\bxi{{\mbox{\boldmath $\xi$}}}
\nc\bmu{{\mbox{\boldmath $\mu$}}} \nc\bcN{{\mbox{\boldmath $\cal{N}$}}} \nc\bcm{{\mbox{\boldmath $\cal{M}$}}} \nc\blambda{{\mbox{\boldmath
$\lambda$}}}\nc\bnu{{\mbox{\boldmath $\nu$}}}

\makeatletter
\def\section{\def\@secnumfont{\mdseries}\@startsection{section}{1}%
  \z@{.7\linespacing\@plus\linespacing}{.5\linespacing}%
  {\normalfont\scshape\centering}}
\def\subsection{\def\@secnumfont{\bfseries}\@startsection{subsection}{2}%
  {\parindent}{.5\linespacing\@plus.7\linespacing}{-.5em}%
  {\normalfont\bfseries}}
\makeatother

 \nc{\Hom}{\operatorname{Hom}}
  \nc{\mode}{\operatorname{mod}}
\nc{\End}{\operatorname{End}} \nc{\wh}[1]{\widehat{#1}} \nc{\Ext}{\operatorname{Ext}} \nc{\ch}{\text{ch}} \nc{\ev}{\operatorname{ev}}
\nc{\Ob}{\operatorname{Ob}} \nc{\soc}{\operatorname{soc}} \nc{\rad}{\operatorname{rad}} \nc{\head}{\operatorname{head}}

\def\spec{\operatorname{spec}}

 \nc{\Cal}{\cal} \nc{\Xp}[1]{X^+(#1)} \nc{\Xm}[1]{X^-(#1)}
\nc{\on}{\operatorname} \nc{\Z}{{\bold Z}} \nc{\J}{{\cal J}}  \nc{\Q}{{\bold Q}}

\nc{\N}{{\bold N}}  \nc\boa{\bold a} \nc\bob{\bold b} \nc\boc{\bold c} \nc\bod{\bold d} \nc\boe{\bold e} \nc\bof{\bold f} \nc\bog{\bold g}
\nc\boh{\bold h} \nc\boi{\bold i} \nc\boj{\bold j} \nc\bok{\bold k} \nc\bol{\bold l} \nc\bom{\bold m} \nc\bon{\mathbb n} \nc\boo{\bold o}
\nc\bop{\bold p} \nc\boq{\bold q} \nc\bor{\bold r} \nc\bos{\bold s} \nc\boT{\bold t} \nc\boF{\bold F} \nc\bou{\bold u} \nc\bov{\bold v}
\nc\bow{\bold w} \nc\boz{\bold z}\nc\ba{\bold A} \nc\bb{\bold B} \nc\bc{\mathbb C} \nc\bd{\bold D} \nc\be{\bold E} \nc\bg{\bold
G} \nc\bh{\bold H} \nc\bi{\bold I} \nc\bj{\bold J} \nc\bk{\bold K} \nc\bl{\bold L} \nc\bm{\bold M} \nc\bn{\mathbb N} \nc\bo{\bold O} \nc\bp{\bold
P} \nc\bq{\bold Q} \nc\br{\bold R} \nc\bs{\bold S} \nc\bt{\bold T} \nc\bu{\bold U} \nc\bv{\bold V} \nc\bw{\bold W} \nc\bz{\mathbb Z} \nc\bx{\bold
x} \nc\KR{\bold{KR}} \nc\rk{\bold{rk}} \nc\het{\text{ht }}

\nc\toa{\tilde a} \nc\tob{\tilde b} \nc\toc{\tilde c} \nc\tod{\tilde d} \nc\toe{\tilde e} \nc\tof{\tilde f} \nc\tog{\tilde g} \nc\toh{\tilde h}
\nc\toi{\tilde i} \nc\toj{\tilde j} \nc\tok{\tilde k} \nc\tol{\tilde l} \nc\tom{\tilde m} \nc\ton{\tilde n} \nc\too{\tilde o} \nc\toq{\tilde q}
\nc\tor{\tilde r} \nc\tos{\tilde s} \nc\toT{\tilde t} \nc\tou{\tilde u} \nc\tov{\tilde v} \nc\tow{\tilde w} \nc\toz{\tilde z} \nc\woi{w_{\omega_i}}

\newcommand\wrapped[1]%
{%
	\begin{array}{@{}l@{}}#1\end{array}%
}

\begin{document}


\title{A generalization of Fiedler's lemma and the spectra of $H$-join of graphs}
\author{M. Saravanan}
\address{Madurai Kamaraj University Constituent College, Sattur, India.}
\email{dr.msaravanan8187@gmail.com.}
\author{S. P. Murugan}
\address{Indian Institute of Science Education and Research, Mohali, India.}
\email{spmath000@gmail.com.}
\author{G. Arunkumar}
\address{Indian Institute of Science, Bangalore, India.}
\email{arun.maths123@gmail.com, garunkumar@iisc.ac.in.}

\thanks{The authors would like to thank M. Rajesh Kannan, Department of Mathematics, Indian Institute of Technology, Kharagpur, for his valuable comments and suggestions on this work. The first author would like to thank him  for the support and the fruitful discussions during his visit to IIT Kharagpur, which is a motivation for this work. The second author acknowledges the institute postdoctoral fellowship of IISER, Mohali. The third author is grateful to Apoorva Khare, Department of Mathematics, Indian Institute of Science, Bangalore, for his constant support and encouragement. The third author also acknowledges the NBHM grant (0204/7/2019/R\&D-II/6831).}


\begin{abstract}
A new generalization of Fiedler's lemma is obtained by introducing the concept of the main function of a matrix. As applications, the universal spectra of the $H$-join, the spectra of the $H$-generalized join and the spectra of the generalized corona of any graphs (possibly non-regular) are obtained.
\end{abstract}

\maketitle

\section{Introduction}

All the graphs considered in this paper are finite and simple. The eigenvalues of a graph $G$ are the eigenvalues of its adjacency matrix $A(G)$. The set of all eigenvalues of $G$ is called the spectrum of $G$, denoted by $\spec(G)$. For more on graphs and their eigenvalues we refer \cite{cvet,cvet2}. Let $H$ be a graph with vertex set $\{v_1,\dots,v_k\}$ and let $\mathcal{F}=\{G_1,G_2, \dots, G_k\}$ be a family of graphs. In \cite{hjn}, the $H$-join operation of the graphs $G_1, G_2, \dots, G_k,$ denoted by $\displaystyle \bigvee_{H}\mathcal{F}$, is obtained by replacing the vertex $v_i$ of $H$ by the graph $G_i$ for $1 \le i \le k$ and every vertex of $G_i$ is made adjacent with every vertex of $G_j$, whenever $v_i$ is adjacent to $v_j$ in $H$. Precisely, $\displaystyle \bigvee_{H}\mathcal{F}$ is the graph with vertex set $V\big(\displaystyle \bigvee_{H} \mathcal{F} \big) = \displaystyle \bigcup_{i=1}^k V(G_i)$ and 
edge set $E\big(\displaystyle \bigvee_{H} \mathcal{F} \big) = \big(\displaystyle  \bigcup_{i=1}^k E(G_i)\big) \cup ( \displaystyle  \bigcup_{ v_iv_j \in E(H)} \{ xy : x \in V(G_i), y \in V(G_j) \}).$
In addition, by considering a family of vertex subsets $\mathcal{S}=\{S_1,S_2, \dots, S_k\}$ where $S_i \subset V(G_i)$ for each $1 \le i \le k$, a generalization of $H$-join operation, known as $H$-generalized join operation constrained by vertex sets, $\displaystyle \bigvee_{H,\mathcal{S}}\mathcal{F}$ is introduced in \cite{cdo13} as follows:
$V\big(\displaystyle \bigvee_{H,\mathcal{S}}\mathcal{F} \big) = \displaystyle \bigcup_{i=1}^k V(G_i)$ and 
$E\big(\displaystyle \bigvee_{H,\mathcal{S}}\mathcal{F} \big) = \big(\displaystyle  \bigcup_{i=1}^k E(G_i)\big) \cup ( \displaystyle  \bigcup_{ v_iv_j \in E(H)} \{ xy : x \in S_i, y \in S_j \}).$ For instance consider the examples in Section \ref{examples}. If we take $S_i = V(G_i)$ for each $1 \le i \le k$, then the $H$-generalized join operation $\displaystyle \bigvee_{H,\mathcal{S}}\mathcal{F}$ coincides with the $H$-join operation of the graphs $G_1,G_2\dots,G_k$. In \cite{sch}, the $H$-join operation of the graphs was initially introduced as generalized composition by Schwenk, denoted by $H[G_1,G_2,\dots,G_k]$. Also, the same operation is studied in some other names as generalized lexicographic product and joined union in \cite{wng,nupa,stev}. When all $G_i$'s are equal to the same graph $G$, it is called the lexicographic product\cite{gdry}, denoted by $H[G]$.

The following lemma \cite[Lemma 2.2]{fid} is proved by M. Fiedler and effectively used in the study of finding sufficient conditions for $k$ arbitrary real numbers to be eigenvalues of a non-negative $k \times k$ symmetric matrix.
\begin{lem}\cite{fid}\label{clasc}
	Let $A$ be a symmetric $m \times m$ matrix with eigenvalues $\alpha_1, \alpha_2,\dots,\alpha_m$ and $B$ be a symmetric $n \times n$ matrix with eigenvalues $\beta_1,\beta_2, \dots,\beta_n$.  let $u$ be an eigenvector of $A$ corresponding to $\alpha_1$ and $v$ be an eigenvector of $B$ corresponding to $\beta_1$ such that $\Vert u \Vert= \Vert v \Vert = 1$. Then for any constant $\rho$ the matrix $$C = \begin{bmatrix}
	A & \rho u v^t \\ \rho v u^t & B
	\end{bmatrix}$$ has eigenvalues $\alpha_2,\dots,\alpha_m,\beta_2,\dots,\beta_n,\gamma_1,\gamma_2$ where $\gamma_1$ and $\gamma_2$ are the eigenvalues of the matrix $$\widehat{C} = \begin{bmatrix}
	\alpha_1 & \rho \\ \rho & \beta_1
	\end{bmatrix}.$$
\end{lem} 

In \cite{cdo11, hjn, cdo13}, the above lemma is called Fiedler's lemma and it has been used to obtain the eigenvalues of some graphs. In \cite{cdo11}, a generalization of the Fiedler's Lemma is obtained \cite[Lemma 2]{cdo11} by Cardoso et al. which can be applied in the $H$-join of regular graphs when $H=P_k$, path on $k$ vertices. Then in \cite{hjn} Cardoso et al. obtained another generalization of Fiedler's lemma \cite[Theorem 3]{hjn} as follows, which can be applied in the $H$-join of regular graphs for any $H$.

 \begin{thm}\cite{hjn}\label{gfiedler}
	Let $M_i$ be a symmetric matrix of order $n_i$ and $u_i$ be an eigenvector  of $M_i$ corresponding to the eigenvalue $\alpha_i$, such that $\Vert u_i \Vert=1$ for $1 \le i \le k$. Let $\rho_{ij}$ be a collection of  arbitrary scalars such that $\rho_{ij} = \rho_{ji}$ for $1 \le i < j \le k$. Considering $$\bold M=(M_1, M_2, \dots, M_k), \bold u=(u_1, u_2, \dots, u_k)$$ as $k$-tuples, and $$ \rho=(\rho_{12}, \dots, \rho_{1k},\rho_{23}, \dots, \rho_{2k}, \dots  \rho_{k-1 k})$$ as $\frac{k(k-1)}{2}$-tuple, the following matrices are defined.\\  
	$A(\bold M, \bold u, \rho):=$
		$\begin{bmatrix}	
	M_1 & \rho_{12} u_1u_2^t & \cdots & \rho_{1k} u_1u_k^t \\
	\rho_{21} u_2u_1^t & M_2 & \cdots & \rho_{2k} u_2u_k^t \\
	\vdots & \vdots & \ddots & \vdots  \\
	\rho_{k1} u_ku_1^t & \rho_{k2} u_ku_2^t & \cdots & M_k
\end{bmatrix}$  and
		  $\widetilde A(\bold M, \bold u, \rho):=$  
    $\begin{bmatrix}
	\alpha_{1} & \rho_{12} & \cdots & \rho_{1k} \\  
	\rho_{21} & \alpha_{2} & \cdots & \rho_{2k}  \\
	\vdots & \vdots & \ddots & \vdots  \\
	\rho_{k1} & \rho_{k2} & \cdots &\alpha_{k}
\end{bmatrix}.$

 Then $\spec(A(\bold M, \bold u, \rho))= \Bigg(\bigcup_{i=1}^k \left( \spec(M_i)\backslash\{\alpha_{i}\} \right)\Bigg) \cup \spec(\widetilde A(\bold M, \bold u, \rho))$  
 \end{thm}

In \cite{hjn} this result is extensively used to compute the eigenvalues of $H$-join of graphs  when the graphs $G_i$'s are regular and in \cite{cdo13} to compute the eigenvalues of $H$-generalized join operation when the subsets $S_i$'s are $(k,\tau)$-regular. 

In this paper,  we obtain a new generalization of the Fiedler's lemma, in terms of characteristic polynomials. The main difference is we are not restricting $\bold u$ as the $k$-tuple of eigenvectors and $\bold M$ as the $k$-tuple of symmetric matrices. But, we consider $\bold u$ as the $k$-tuple of any complex vectors and $\bold M$ as the $k$-tuple of any complex square matrices of appropriate size. We accomplish this task by introducing the concept of the main function of a matrix, in Section \ref{secmf}. Also as an application of our result, we obtain the characteristic polynomial of $H$-join of graphs when the graphs $G_i$'s are any graphs(possibly non-regular). In \cite{sch} it is remarked by Schwenk, that ``In general, it does not appear likely that the characteristic polynomial of the generalized composition can always be expressed in terms of the characteristic polynomials of $H,G_1,G_2,\dots,G_k$". In our paper, we prove that it is possible to express the characteristic polynomial of $H$-join operation of graphs (i.e. generalized composition) in terms of the characteristic polynomials and main functions of $G_1,G_2,\dots,G_k$, and another function obtained from the adjacency matrix of $H$. Moreover for the $H$-join operation of any graphs, we obtain the characteristic polynomial of its universal adjacency matrix. 

The universal adjacency matrix of a graph $G$ is defined as follows: Let $A(G)$, $I$, $J$, and $D(G)$ be the adjacency matrix of $G$, the identity matrix, the all-one matrix, and the degree matrix of $G$, respectively. Any matrix of the form $U(G) = \alpha A + \beta I + \gamma J + \delta D$ where $\alpha,\beta,\gamma,\delta \in \mathbb R$ and $\alpha \ne 0$ is called the universal adjacency matrix of $G$. Many interesting and important matrices associated to a graph can be obtained as special cases of $U(G)$. For example, from the universal adjacency matrix $U(G)$, we get adjacency matrix $A(G)$, Laplacian matrix $L(G)=D(G)-A(G)$, signless  Laplacian matrix $Q(G)=D(G)+A(G)$, and Seidel matrix $S(G)=J-I-2A(G)$ by taking appropriate values for $\alpha, \beta, \gamma$, and $\delta$. 

In \cite{gerb}, the Laplacian spectra of $H$-join of any graphs is obtained. In \cite{chen19} the characteristic polynomial of the matrix $A(G)-tD(G)$ is obtained for $H$-join of regular graphs. In \cite{wng} the characteristic polynomial of the adjacency matrix of the lexicographic product of any graphs is obtained. Recently in \cite{hmob} universal adjacency spectra of the disjoint union of regular graphs is obtained.

In our paper, we obtain the characteristic polynomial and eigenvalues of the universal adjacency matrix of $H$-join of any graphs $G_1, G_2, \dots, G_k$. Then we obtain the characteristic polynomial and eigenvalues of the adjacency matrix of $H$-generalized join of graphs $G_1, G_2, \dots, G_k$, where the subsets $S_i(G) \subset V(G_i)$ are arbitrary  for $1 \le i \le k$. Also, we deduce the characteristic polynomial of the generalized corona of graphs by visualizing corona as  $H$-join of graphs. Hence many results obtained (mostly for regular graphs) in \cite{hjn, cdo13, chen19, lsk, hmob, wng, wu14}, are generalized here for any graphs.

Throughout this paper, we denote the identity matrix of order $n$ by $I_{n}$, the all-one matrix of order $n$ by $J_{n}$ and the all-one vector of size $n \times 1$ by $\textbf{1}_{n}$. 

\section{The main function of a matrix} \label{secmf}
Consider a graph $G$ on $n$ vertices with adjacency matrix $A(G)$. Suppose $A(G)$ has spectral decomposition $A(G)=\Sigma_{i=1}^k \theta_i E_i$, where $\theta_i$'s are distinct eigenvalues of $G$ and $E_i$ is the orthogonal projection on the eigenspace of $\theta_i$. An eigenvalue $\theta_i$ is called a main eigenvalue\cite{rowl} if the corresponding eigenspace $\mathcal{E}(\theta_i)$ is not orthogonal to $\textbf{1}_n$. The cosines of the angles between $\textbf{1}_n$ and the eigenspaces of $A$ are known as main angles of $G$, given by $\beta_i=\dfrac{1}{\sqrt{n}}\Vert E_i \textbf{1}_n\Vert$, for $1\le i \le k$. So $\theta_i$ is a main eigenvalue if and only if $\beta_i \ne 0$.  

Consider the field of rational functions $\mathbb C(\lambda)$. The $\det(\lambda I -A)$ is a non-zero element of $\mathbb C(\lambda)$ and hence the matrix $\lambda I - A$ is invertible over $\mathbb C(\lambda).$ In \cite{mcln}, the function $\textbf{1}_n^T(\lambda I_n-A(G))^{-1}\textbf{1}_n$ is introduced in the name of coronal of $G$ and is used to find the characteristic polynomial of the corona of two graphs. Since $E_i^2=E_i$, it is easy to see that\\ $$\textbf{1}_n^T(\lambda I_n-A(G))^{-1}\textbf{1}_n=\Sigma_{i=1}^k \dfrac{\textbf{1}_n^T E_i \textbf{1}_n}{\lambda-\theta_i} =\Sigma_{i=1}^k \dfrac{ \Vert E_i \textbf{1}_n \Vert^2}{\lambda-\theta_i}= \Sigma_{i=1}^k \dfrac{n\beta_i^2}{\lambda-\theta_i},$$ in which only non-vanishing terms are those terms corresponding to main eigenvalues.

Also in \cite{mcln}, the authors remarked that graphs with different eigenvalues can have the same coronals whereas cospectral graphs can have different  coronals. This is because of the fact that the main function of a graph depends not only on the eigenvalues but also on the main angles of the graph. For more on the main angles and main eigenvalues, we refer \cite{rowl} and references therein. Because of these relationships with main eigenvalues and main angles of the graph $G$, in this paper we call this function $\textbf{1}_n^T(\lambda I_n-A(G))^{-1}\textbf{1}_n$, the main function of the graph $G$. Moreover for any vectors $u$ and $v$, and a matrix $M$ of the same dimension, we introduce the following notions.
\begin{defn}
	Let $M$ be an $n \times n$ complex matrix, and let $u$ and $v$ be $ n \times 1$ complex vectors. The main function associated to the matrix $M$ corresponding to the vector $u$ and $v$, denoted by $\Gamma_M(u,v)$, is defined to be $\Gamma_M(u,v;\lambda) = v^{t}(\lambda I - M)^{-1}u \in \mathbb C(\lambda)$. When $u=v$, we denote $\Gamma_M(u,v;\lambda)=\Gamma_M(u;\lambda).$ 
\end{defn}

\begin{defn}
	Let $M$ be an $n \times n$ normal matrix over $\mathbb{C}$ and let $u$ be an $ n \times 1$ complex vector. An eigenvalue $\lambda$ of $M$ is called as $u$-main eigenvalue if $u$ is not orthogonal to the eigenspace $\mathcal E_M(\lambda)$. In the case of $u=\textbf{1}_n$, the all-one vector, then we don't specify the vector and call eigenvalue $\lambda$ of $M$ as the main eigenvalue of $M$.  
\end{defn}

\begin{lem}\label{evmain}
	Let $M$ be a matrix of order $n$ with an eigenvector $u$ corresponding to the eigenvalue $\mu$. Then $\Gamma_M(u;\lambda) = \dfrac{\Vert u \Vert^2}{\lambda-\mu}$.
\end{lem}
\begin{pf}
	
	Now $(\lambda I_n - M)u=(\lambda-\mu)u$. Applying $(\lambda I_n - M)^{-1}$ both sides, we can get $u=(\lambda I_n - M)^{-1}(\lambda-\mu)u$, which implies $\dfrac{u^T u}{\lambda-\mu}=u^T (\lambda I_n - M)^{-1}u=\dfrac{\Vert u \Vert^2}{\lambda-\mu}$.
\end{pf}

\begin{lem}\label{polymain}
	Let $M$ be an $n \times n$ normal matrix over $\mathbb{C}$, $u$ be an $ n \times 1$ complex vector and $p(M)$ be a polynomial in $M$ with complex coefficients. Then an eigenvalue $\mu$ is a $u$-main eigenvalue of $M$ if and only if $p(\mu)$ is a $u$-main eigenvalue of $p(M)$. 
\end{lem}
\begin{pf}
 For any eigenvalue of $M$ and corresponding eigenvalue of $p(M)$ the eigenvectors are the same. So the eigenspaces are the same and hence the result follows.   
 \end{pf}

Now we can state our main result, a new generalization of Fiedler's lemma. 

 \begin{thm}\label{mainthm}
	
	Let $M_i$ be a complex matrix of order $n_i$, and let $u_i$ and $v_i$ be arbitrary complex vectors of size $n_i \times 1$ for $1 \le i \le k$. Let $\rho_{ij}$ be arbitrary complex numbers for $1 \le i,j \le k$ and $i \ne j$. For each $1 \le i \le k$, let $\phi_i(\lambda)=\det(\lambda I_{n_i}-M_i)$ be the characteristic polynomial of the matrix $M_i$ and $\Gamma_i(\lambda) = \Gamma_{M_i}(u_i,v_i;\lambda) = v_i^t (\lambda I - M_i)^{-1} u_i$. 
	Considering 
\begin{center}
the $k$-tuple $\bold M=(M_1, M_2, \dots, M_k)$, 2$k$-tuple $\bold u=(u_1,v_1, u_2,v_2 \dots, u_k,v_k)$ \\ and ${k(k-1)}$-tuple $ \rho=(\rho_{12}, \rho_{12} \dots, \rho_{1k},\rho_{21}, \rho_{23}, \dots, \rho_{2k}, \dots, \rho_{k1}, \rho_{k2}, \dots,  \rho_{k-1 k})$
\end{center}	
	 the following matrices are defined: $$A(\bold M, \bold u, \rho) := \begin{bmatrix}
	
	M_1 & \rho_{12} u_1v_2^t & \cdots  & \rho_{1k} u_1v_k^t \\
	\rho_{21} u_2v_1^t & M_2 & \cdots  & \rho_{2k} u_2v_k^t \\
	\vdots & \vdots & \ddots  & \vdots \\
	\rho_{k1} u_kv_1^t & \rho_{k2} u_kv_2^t & \cdots  &M_k
\end{bmatrix}$$ $$\text{ and }
 \widetilde{A}(\bold M, \bold u, \rho) :=  \begin{bmatrix}
\frac{1}{\Gamma_1} & -\rho_{12}  & \cdots & -\rho_{1,k}  \\
-\rho_{21}  & \frac{1}{\Gamma_2} & \cdots & -\rho_{2,k} \\
\vdots & \vdots & \ddots & \vdots  \\
-\rho_{k1} & -\rho_{k2} & \cdots &\frac{1}{\Gamma_{k}}
\end{bmatrix}.$$

Then the characteristic polynomial of $A(\bold M, \bold u, \rho)$ is given as
\begin{equation}\label{maineqn}
\det(\lambda I - A(\bold M, \bold u, \rho)) = \Bigg( \Pi_{i=1}^k \phi_i(\lambda) \Gamma_i(\lambda) \Bigg) \det(\widetilde{A}(\bold M, \bold u, \rho)) .
\end{equation} 
\end{thm}
Proof of this theorem is given in Section \ref{subsecpfofmainthm}. At first, we deduce Theorem \ref{gfiedler} in terms of characteristic polynomials as a corollary of Theorem \ref{mainthm}.

\begin{cor}
Consider the notations defined in Theorem \ref{mainthm}. Suppose $u_i=v_i$ is an eigenvector of $M_i$ corresponding to an eigenvalue $\alpha_i$ with $\Vert u_i \Vert=1$, then the characteristic polynomial of $A(\bold M, \bold u, \rho)$  is \begin{center}
 $\phi(A(\bold M, \bold u, \rho))=\frac{\phi_1}{\lambda-\alpha_1} \frac{\phi_2}{\lambda-\alpha_2} \dots \frac{\phi_k}{\lambda-\alpha_k} \det(\widetilde A (\bold M, \bold u, \rho))$ 
 
 where $\widetilde A (\bold M, \bold u, \rho)=
    \begin{bmatrix}
	\lambda-\alpha_1 & -\rho_{12} & \cdots & -\rho_{1k} \\  
	-\rho_{21} & \lambda-\alpha_2 & \cdots & -\rho_{2k}  \\
	\vdots & \vdots & \ddots & \vdots  \\
	-\rho_{k1} & -\rho_{k2} & \cdots &\lambda-\alpha_k
\end{bmatrix}$.
 \end{center}
 \begin{proof}
 Since $\Vert u_i \Vert=1$, by Lemma \ref{evmain} we get $\Gamma_i=\dfrac{1}{\lambda-\alpha_i}.$ Now the proof follows from Theorem \ref{mainthm}.
 \end{proof}
\end{cor}

In \cite[Theorem 2.3]{chen19} another generalization of Fiedler's lemma, similar to Theorem \ref{gfiedler}, is given for the matrices with fixed row sum and the result is used to find the generalized characteristic polynomial of $H$-join of regular graphs. We observe that any such matrix has the all-one vector as an eigenvector. By taking $u_i$ to be the all-one vector of appropriate size, we can deduce \cite[Theorem 2.3]{chen19} from Theorem \ref{mainthm}.

\section{Proof of the main result}
In this section, we prove Theorem \ref{mainthm}. We start with the following essential lemmas.
\subsection{Some important lemmas}
	\begin{lem}\label{lemd}\cite{cvet}
	Let $A, B, C$ and $D$ be matrices such that $M=
	\begin{bmatrix}
	A & B\\
	C & D
	\end{bmatrix}
	$.  If $D$ is invertible, then 
	 $$\det(M) = \det(D) \det(A - BD^{-1}C).$$

\end{lem}
\begin{lem}\cite{bart,dizh}\label{lemid}
	Let $A$ be an $n \times n$ invertible matrix, and let $u$ and $v$ be any two $n \times 1$ vectors such that $1 + v^{t}A^{-1}u \ne 0$. Then 
	\begin{enumerate}
		\item $\det (A+uv^t) = (1+ v^t A^{-1} u) \det (A).$
		\item $(A+uv^t)^{-1} = A^{-1} - \dfrac{A^{-1} u v^t A^{-1}}{1+v^tA^{-1}u}.$
	\end{enumerate} 
\end{lem}
\begin{lem}\label{lemid2}
	Let $A$ be an $n \times n$ complex matrix, and let $u$ and $v$ be any $n \times 1$ complex vectors. Also, Let $\Gamma = v^t (\lambda I - A)^{-1} u$. Then 
	\begin{enumerate}
		\item $\det (\lambda I - A + \alpha uv^t) = (1 + \alpha \Gamma) \det(\lambda I - A) = (1 + \alpha \Gamma) \phi_A(\lambda) $	
		\item $v^t(\lambda I - A + \alpha uv^t)u= \dfrac{\Gamma}{1 +\alpha \Gamma}$ 
	\end{enumerate} 
\end{lem}
\begin{proof} The proof of (1) follows directly from Lemma \ref{lemid}(1), as $\det (\lambda I - A + \alpha uv^t)=(1+\alpha v^T(\lambda I - A)^{-1}u)\det(\lambda I - A)$. So we prove (2).\\
By Lemma \ref{lemid}(2),
$$(\lambda I - A + \alpha uv^t)^{-1}=(\lambda I - A)^{-1}-\alpha \dfrac{(\lambda I - A)^{-1}uv^T(\lambda I - A)^{-1}}{1+\alpha v^T(\lambda I - A)^{-1}u}$$
which implies,
$$v^T(\lambda I - A + \alpha uv^t)^{-1}u=\Gamma-\alpha \dfrac{\Gamma^2}{1+\alpha \Gamma}=\dfrac{\Gamma}{1 +\alpha \Gamma}$$
\end{proof}
Motivated by \cite[Theorem 8.13.3]{gdry}, we prove the following lemma.
\begin{lem}\label{uev}
	Let $M$ be a complex normal matrix of order $n$ and let $u$ be any $n \times 1$ vector. Then the poles of $u^T(\lambda I - M)^{-1}u$ are the $u$-main eigenvalues of $M$ and are simple.
\end{lem}
\begin{pf}
Let $\{\theta_1, \theta_2, \cdots, \theta_{k}\}$ be the distinct eigenvalues and let $\{\theta_1,\theta_2,\dots,\theta_{m}\}$ be the set of $u$-main eigenvalues of $M$.

	Suppose the spectral decomposition of $M$ is $M=\Sigma_{j=1}^{k} \theta_j E_{\theta_j}$, where $E_{\theta_j}$ is the orthogonal projection on the eigenspace of $\theta_j$. Then $(\lambda I - M)^{-1}=\Sigma_{j=1}^{k} \frac{E_{\theta_j}}{\lambda-\theta_j}$, and $\Gamma_{M}(u) = u^T(\lambda I - M)^{-1}u=\Sigma_{j=1}^{k} \frac{u^tE_{\theta_j}u}{\lambda-\theta_j}.$ Now, $u^t E_{\theta_j} u \ne 0$ if and only if $\theta_j$ is a $u$-main eigenvalue of $M$. So, $\Gamma_{M}(u)=\Sigma_{j=1}^{m} \frac{u^tE_{\theta_j}u}{\lambda-\theta_j}$ and the result follows.
\end{pf}
\subsection{Proof of Theorem \ref{mainthm}}\label{subsecpfofmainthm}
\begin{proof}
We prove the result by using induction on $k$.
The base case $k=1$ is clear. We prove the result also for $k=2$ for the sake of understanding. 	Now, by Lemma \ref{lemd}, we have
\begin{align*}
\begin{vmatrix}
\lambda I_{n_1}-M_1 & -\rho_{12}u_1v_2^t \\
-\rho_{21}u_2v_1^t & \lambda I_{n_2}- M_2
\end{vmatrix} &= \det(\lambda I_{n_2}- M_2) \det (\lambda - M_1 - \rho_{12}\rho_{21} \Gamma_2 u_1 v_1^t)\\
&= \phi_1 \phi_2 (1-\rho_{12}\rho_{21}\Gamma_2\Gamma_1), \text{by Lemma } \ref{lemid2}(1)\\
&= \phi_1 \phi_2 \begin{vmatrix}
1 & -\rho_{12}\Gamma_1 \\
-\rho_{21}\Gamma_2 & 1
\end{vmatrix}
\end{align*}

This proves the result for the case $k=2$. We assume the result is true for $k-1$. Let $$M=\begin{bmatrix}
	\lambda I_{n_1} - 
	M_1 & -\rho_{12} u_1v_2^t & \cdots  & -\rho_{1k} u_1v_k^t \\
	-\rho_{21} u_2v_1^t & \lambda I_{n_2}-M_2 & \cdots & -\rho_{2k} u_2v_k^t  \\
	\vdots  & \vdots  & \ddots   & \vdots\\
	-\rho_{k1} u_kv_1^t & -\rho_{k2} u_kv_2^t & \cdots \cdots  &\lambda I_{n_k}-M_k
	\end{bmatrix}$$ \\
	Now, by Lemma \ref{lemd}, we have
\begin{equation} \label{pfind}
	\det(M)=\det(\lambda I_{n_k} - M_k)\det(S)
\end{equation}
	$\text{where } S=\begin{bmatrix}
	\lambda I_{n_1} -	M_1 & -\rho_{12} u_1v_2^t & \cdots & -\rho_{1,k-1} u_1v_{k-1}^t \\
	-\rho_{21} u_2v_1^t & \lambda I_{n_2}-M_2 & \cdots & -\rho_{2,k-1} u_2v_{k-1}^t \\
	\vdots & \vdots & \ddots & \vdots  \\
	-\rho_{k-1,1} u_{k-1}v_1^t & -\rho_{k-1,2} u_{k-1}v_2^t & \cdots & \lambda I_{n_{k-1}}-M_{k-1} &
	\end{bmatrix}\\ -\begin{bmatrix}
	-\rho_{1k} u_1v_k^t \\
	-\rho_{2k} u_2v_k^t \\
	\vdots  \\
	-\rho_{k-1,k} u_{k-1}v_k^t
	\end{bmatrix} (\lambda I_{n_k} - M_k)^{-1}
	\begin{bmatrix}
	-\rho_{k1} u_kv_1^t & -\rho_{k2} u_kv_2^t & \cdots & -\rho_{k,k-1} u_{k}v_{k-1}^t
	\end{bmatrix}$\\
	
	$=[s_{ij}]$
	given by $s_{ij}=\begin{cases} \lambda I_{n_i}-M_i - \Gamma_k  \rho_{ik} \rho_{ki} u_iv_i^t \text{ if }i=j\\
	-(\rho_{ij} + \Gamma_k  \rho_{ik} \rho_{kj})  u_iv_j^t \text{ if }i\ne j
	 \end{cases}$	
	
 By Lemma \ref{lemid2},
$\det(
	\lambda I_{n_i}-M_i - \Gamma_k  \rho_{ik} \rho_{ki} u_iv_i^t)=\det(
	\lambda I_{n_i}-M_i)(1-(\Gamma_k  \rho_{ik} \rho_{ki})\Gamma_i)$ and\\ $v_i^t (
	\lambda I_{n_i}-M_i - \Gamma_k  \rho_{ik} \rho_{ki} u_iv_i^t)^{-1} u_i = \dfrac{\Gamma_i}{1 - (\Gamma_k\rho_{ik} \rho_{ki})\Gamma_i}.$
  
	 Applying these results on the induction hypothesis on $S$ we get
\begin{align*}
\det(S)&=\Pi_{i=1}^{k-1}\Bigg(\det( \lambda I_{n_i}-M_i - \Gamma_k  \rho_{ik} \rho_{ki} u_iv_i^t )\dfrac{\Gamma_i}{1 - (\Gamma_k\rho_{ik} \rho_{ki})\Gamma_i}\Bigg)\det(\widetilde{S})\\
&=\Pi_{i=1}^{k-1}\Bigg(\det( \lambda I_{n_i}-M_i)\Gamma_i \Bigg)\det(\widetilde{S})
\end{align*}	 
	  where $\widetilde{S}=[\widetilde{s}_{ij}]$ 	given by $\widetilde{s}_{ij}=\begin{cases} \dfrac{1 - \rho_{ik} \rho_{ki} \Gamma_k\Gamma_i}{\Gamma_i} \text{ if }i=j\\
	-(\rho_{ij} + \Gamma_k  \rho_{ik} \rho_{kj}) \text{ if }i\ne j
	 \end{cases}$
	 
	Therefore $\det(S)=\phi_1 \phi_2 \cdots \phi_{k-1} \times \\
	\begin{vmatrix}
	1-(\rho_{1k}\rho_{k1}\Gamma_k \Gamma_1) & -(\rho_{12} + \Gamma_k  \rho_{1,k} \rho_{k2}) \Gamma_1 &\cdots& -(\rho_{1k-1}+ \Gamma_k  \rho_{1,k}\rho_{k,k-1}) \Gamma_1 \\
	-(\rho_{21}  + \Gamma_k  \rho_{2,k} \rho_{k1}) \Gamma_2 & 1-(\rho_{2k}\rho_{k2}\Gamma_k \Gamma_2) &\cdots& -(\rho_{2k-1} + \Gamma_k  \rho_{2,k}\rho_{k,k-1}) \Gamma_2  \\
	\vdots & \vdots & \ddots & \vdots \\
	-(\rho_{k-1,1} + \Gamma_k  \rho_{k-1,k}\rho_{k1}) \Gamma_{k-1} & -(\rho_{k-1,2} + \Gamma_k \rho_{k-1,k}\rho_{k2}) \Gamma_{k-1}&\cdots  & 1-(\rho_{k-1,k}\rho_{k,k-1}\Gamma_k \Gamma_{k-1})
	\end{vmatrix}$
		$=\phi_1 \phi_2 \cdots \phi_{k-1} \times \bigg( 
	\begin{vmatrix}
	1 & -\rho_{12} \Gamma_1 &\cdots& -\rho_{1k-1} \Gamma_1 \\
	-\rho_{21}  \Gamma_2 & 1 &\cdots& -\rho_{2k-1} \Gamma_2  \\
	\vdots & \vdots & \vdots  \\
	-\rho_{k-1,1} \Gamma_{k-1} & -\rho_{k-1,2} \Gamma_{k-1}&\cdots  & 1
	\end{vmatrix}$ \\
	$- \begin{vmatrix}
	-\rho_{1k}\rho_{k1}\Gamma_k \Gamma_1 & -   \rho_{1,k} \rho_{k2} \Gamma_k\Gamma_1 &\cdots& - \rho_{1,k}\rho_{k,k-1} \Gamma_k\Gamma_1 \\
	- \rho_{2,k} \rho_{k1} \Gamma_k \Gamma_2 & - \rho_{2k}\rho_{k2}\Gamma_k \Gamma_2 &\cdots& - \rho_{2,k}\rho_{k,k-1} \Gamma_k \Gamma_2  \\
	\vdots & \vdots & \vdots  \\
	-  \rho_{k-1,k}\rho_{k1} \Gamma_k \Gamma_{k-1} & -  \rho_{k-1,k}\rho_{k2} \Gamma_k \Gamma_{k-1}&\cdots  & -\rho_{k-1,k}\rho_{k,k-1}\Gamma_k \Gamma_{k-1}
	\end{vmatrix} \bigg)\\
	=\phi_1 \phi_2 \cdots \phi_{k-1} \times 
	\begin{vmatrix}
	1 & -\rho_{12} \Gamma_1 &\cdots & -\rho_{1,k-1} \Gamma_1 & -\rho_{1k} \Gamma_1 \\
	-\rho_{21} \Gamma_2 & 1 &\cdots& -\rho_{2,k-1} \Gamma_2 & -\rho_{2k} \Gamma_2  \\
	\vdots & \vdots & \vdots & \vdots & \vdots\\
	-\rho_{k-1,1} \Gamma_{k-1} & -\rho_{k-1,2} \Gamma_{k-1}&\cdots  & 1 & -\rho_{k-1,k} \Gamma_{k-1} \\
	-\rho_{k,1} \Gamma_{k} & -\rho_{k,2} \Gamma_{k}&\cdots  & -\rho_{k,k-1} \Gamma_k & 1
	\end{vmatrix}$, again by Lemma \ref{lemd}.

By substituting this $\det(S)$ value in Equation \eqref{pfind}, we get the required result for $k$. 
This completes the proof of Theorem \ref{mainthm}. 
\end{proof}

Suppose the matrices $M_i$'s are normal and $\{\theta_1,\theta_2,\dots,\theta_{m_i}\}$ is the set of distinct $u_i$-main eigenvalues of $M_i$, for $1 \le i \le k$. Then as discussed in the proof of Lemma \ref{uev}, we can write 
\begin{equation}\label{defnfg}
\Gamma_i=\dfrac{f_i}{g_i}  \text{ where } g_i=\prod_{j=1}^{m_i}(\lambda - \theta_j).
\end{equation}
Hence by the Theorem \ref{mainthm},
\begin{equation}\label{maineqn1}
\det(\lambda I - A(\bold M, \bold u, \rho)) = (\frac{\phi_1}{g_1}) \dots (\frac{\phi_k}{g_k}) \Phi(\lambda)
\end{equation} 
where  $\Phi(\lambda)=\begin{vmatrix}
g_1(\lambda) & -\rho_{12}f_1(\lambda)  & \cdots & -\rho_{1,k}f_1(\lambda)  \\
-\rho_{21}f_2(\lambda)  & g_2(\lambda) & \cdots & -\rho_{2,k}f_2(\lambda) \\
\vdots & \vdots & \ddots & \vdots  \\
-\rho_{k1}f_k(\lambda) & -\rho_{k2}f_k(\lambda) & \cdots &g_k(\lambda)
\end{vmatrix} $\\
So we can describe the spectrum of $A(\bold M, \bold u, \rho) $ as follows. 
\begin{thm}\label{mainthm2}
Consider the notations defined above. Suppose the matrices $M_i$'s are normal, then 
\begin{itemize}
\item Every eigenvalue, which is not a $u_i$-main eigenvalue of $M_i$, say $\lambda$ with multiplicity $m(\lambda)$ is an eigenvalue of $A(\bold M, \bold u, \rho) $ with multiplicity $m(\lambda)$.
\item Every $u_i$- main eigenvalue of $M_i$, say $\lambda$ with multiplicity $m(\lambda)$ is an eigenvalue of $A(\bold M, \bold u, \rho) $ with multiplicity $m(\lambda)-1$.
\item Remaining eigenvalues are the roots of the polynomial $\Phi(\lambda).$
\end{itemize}
\end{thm}
\begin{proof}
By Lemma \ref{uev} the poles of $\Gamma_i$ are $u_i$- main eigenvalues and they are simple. Now 
the proof easily follows from Equation \eqref{maineqn1}.
\end{proof}
\section{Universal spectra of the H-join of graphs}\label{secappuni}
In this section, by applying Theorem \ref{mainthm}, we obtain the results on characteristic polynomial and spectrum of the universal adjacency matrix of $H$-join of graphs. 

 Consider a graph $H$ on $k$ vertices and a family of graphs $\mathcal{F}=\{G_1,G_2, \dots, G_k\}$. Let $G=\displaystyle \bigvee_{H}\mathcal{F}$ be the $H$-join of graphs in $\mathcal{F}$, and let $n_i$, $A_i$ and $D_i$  be the number of vertices, the adjacency matrix and the degree matrix of the graph $G_i$ respectively, for $1\le i \le k$. Also let $\rho_{ij}$ be the scalars defined by $\rho_{ij} = \rho_{ji}= 1 \text{ if}$ $ij \in E(H)$ and  $0 \text{ otherwise}$, for $1\le i,j \le k$ and $i \ne j$.  Once and for all we fix an ordering of the vertices of $G$, such that the adjacency matrix of the graph $G$ is given as  
 \begin{equation} \label{AofG}
 A(G) = \begin{bmatrix}
A_1  & \rho_{1,2} \textbf{1}_{n_1}\textbf{1}_{n_2}^t & \cdots & \rho_{1,k} \textbf{1}_{n_1} \textbf{1}_{n_k}^t \\
\rho_{2,1} \textbf{1}_{n_2} \textbf{1}_{n_1}^t & A_2 & \cdots & \rho_{2,k} \textbf{1}_{n_2} \textbf{1}_{n_k}^t \\
\vdots & \vdots & \ddots & \vdots  \\
\rho_{k,1} \textbf{1}_{n_k} \textbf{1}_{n_1}^t & \rho_{k,2} \textbf{1}_{n_k} \textbf{1}_{n_2}^t & \cdots & A_k
\end{bmatrix}.
  \end{equation}

\subsection{Universal spectra of the $H$-join of graphs}\label{subsecushjoin}
The proof of the following lemma is immediate from the definition of the $H$-join of graphs. 
\begin{lem}
	Let $H$ be a graph with vertex set $\{v_1,\dots,v_k\}$ and  $\mathcal{F}=\{G_1,G_2, \dots, G_k\}$ be a family of $k$ graphs such that $V(G_i)=\{v_{1}^{(i)},\dots,v_{n_i}^{(i)}\}$ for $1 \le i \le k$. Then the degree of the vertex $v_{j}^{(i)}$ in $G$ is given by  $$\deg_{G}(v_{j}^{(i)}) = \deg_{G_i}(v_{j}^{(i)}) + w_i, 1 \le i \le k, 1 \le j \le n_i$$ where $w_i = \displaystyle \sum_{v_l \in N_H(v_i)}n_l$ . 
\end{lem}

Let $U(G) = \alpha A(G) + \beta I_n + \gamma J_n + \delta D(G)$ with $\alpha \ne 0$ be the universal adjacency matrix of the graph $G$, where $n=\Sigma_{i=1}^k n_i$. Let $U_i = \alpha A_i + \beta I_{n_i} + \gamma J_{n_i} + \delta D_i$ be the universal adjacency matrix of the graph $G_i$, for $1 \le i \le k$.  Therefore by the Equation \eqref{AofG} the universal adjacency matrix of $G$ can be written as follows:
\begin{equation} \label{UAofG}
U(G)=\begin{bmatrix}
U_1 +  \delta w_1 I_{n_1}  & (\rho_{1,2} \alpha +\gamma)\textbf{1}_{n_1}\textbf{1}_{n_2}^t & \cdots & (\rho_{1,k}\alpha +\gamma)\textbf{1}_{n_1}\textbf{1}_{n_k}^t\\
 (\rho_{2,1}\alpha +\gamma)\textbf{1}_{n_2}\textbf{1}_{n_1}^t & U_2 +  \delta w_2 I_{n_2} & \cdots & (\rho_{2,k}\alpha +\gamma)\textbf{1}_{n_2}\textbf{1}_{n_k}^t\\
\vdots & \vdots & \ddots & \vdots  \\
(\rho_{k,1}\alpha +\gamma)\textbf{1}_{n_k}\textbf{1}_{n_1}^t& (\rho_{k,2}\alpha +\gamma)\textbf{1}_{n_k}\textbf{1}_{n_2}^t& \cdots & U_k +  \delta w_k I_{n_k}
\end{bmatrix}
\end{equation}

In the following theorem, we obtain the characteristic polynomial of universal adjacency matrix $U(G)$. 
\begin{thm}\label{chUofG}
Let $H$ be a graph on $k$ vertices and  $\mathcal{F}=\{G_1,G_2, \dots, G_k\}$ be a family of any $k$ graphs. Consider the graph $G=\displaystyle \bigvee_{H}\mathcal{F}$. Let $\phi_{i}(\lambda)$ be the characteristic polynomial of $U_i$ and $\Gamma_{i}(\lambda)=\Gamma_{U_i}(\textbf{1}_{n_i};\lambda)$. Then we have the following.

i) The characteristic polynomial of the universal adjacency matrix $U(G)$ given in Equation \eqref{UAofG} is
$$\phi_{U(G)}(\lambda) = \displaystyle \Pi_{i=1}^k \phi_i(\lambda-\delta w_i) \Gamma_i(\lambda-\delta w_i) \det(\widetilde{U}(G))$$

\begin{equation}\label{UAdofG}
	\text{where}\,\,    \widetilde{U}(G) =   \begin{bmatrix}
	\frac{1}{\Gamma_1(\lambda-\delta w_1)} & -(\rho_{1,2} \alpha +\gamma)  & \cdots & -(\rho_{1,k} \alpha +\gamma)  \\
	-(\rho_{2,1} \alpha +\gamma)  & \frac{1}{\Gamma_2(\lambda-\delta w_2)} & \cdots & -(\rho_{2,k} \alpha +\gamma) \\
	\vdots & \vdots & \ddots & \vdots  \\
	-(\rho_{k,1} \alpha +\gamma) & -(\rho_{k,2} \alpha +\gamma) & \cdots &\frac{1}{\Gamma_{k}(\lambda-\delta w_k)}
	\end{bmatrix}
\end{equation}			
	
ii) Analogous to the Equations \eqref{defnfg} and \eqref{maineqn1}, we define $f_i, g_i$, and $\Phi(\lambda)$  corresponding to the main eigenvalues of $U_i$ for $1 \le i \le k$. Then the spectrum of $G$ is given as below. 
\begin{itemize}
\item For every eigenvalue $\mu$ of $A_i$ with multiplicity $m(\mu)$, which is not a main eigenvalue,   $\mu+\delta w_i$  is an eigenvalue of $G$ with multiplicity $m(\mu)$.
\item For every main eigenvalue $\mu$ of $A_i$ with multiplicity $m(\mu)$, $\mu+\delta w_i$  is an eigenvalue of $G$ with multiplicity $m(\mu)-1$.
\item Remaining eigenvalues are the roots of the polynomial $\Phi(\lambda)$.
\end{itemize}
\end{thm}
\begin{pf} 
For each $1\le i \le k$, let $P_i=U_i +  \delta w_i I_{n_i}$. Then  we have the following relations,\\
 $$\phi_{P_i}(\lambda)=\det(\lambda I_{n_i}-(U_i +  \delta w_i I_{n_i}))=\phi_{U_i}(\lambda-\delta w_i) \,\,\text{and}$$  $$\Gamma_{P_i}(\lambda)=\textbf{1}_{n_i}^T(\lambda I_{n_i}-(U_i +  \delta w_i I_{n_i}))^{-1}\textbf{1}_{n_i}=\Gamma_{U_i}(\lambda-\delta w_i).$$

Let $\widehat{\rho}_{ij}=\widehat{\rho}_{ji}=\rho_{ij} \alpha +\gamma$ for $1\le i<j \le k$.
Considering the triplet $(\bold M, \bold u, \widehat{\rho} )$, given by $$\bold M = (P_1,P_2,\dots,P_k), \bold u = ( \bold 1_{n_1}, \bold 1_{n_2}\dots, \bold 1_{n_k})\, \text{and}$$
$$\widehat{\rho}=(\widehat{\rho}_{12},\dots,\widehat{\rho}_{1k},\widehat{\rho}_{23},\dots,\widehat{\rho}_{2k},\dots,\widehat{\rho}_{k-1,k} )$$
we can write the matrices in the Equations \eqref{UAofG} and \eqref{UAdofG} as $U(G)=A(\bold M, \bold u, \widehat{\rho})$ and $\widetilde{U}(G)=\widetilde{A}(\bold M, \bold u, \widehat{\rho})$. 
Now using Theorem \ref{mainthm} the proof of (i) follows. By Lemma \ref{polymain}, $\mu$ is not a main eigenvalue of $U_i$ if and only if $\mu + \delta w_i$ is not a main eigenvalue of $P_i$. Now by Theorem \ref{mainthm2} the proof of (ii) follows.
\end{pf}

\begin{cor} \label{specUofG}
Consider the notations defined in Theorem \ref{chUofG}. Suppose the graph $G_i$ is $r_i$-regular for $1 \le i \le k$. Then $p_i=\alpha r_i + \beta + \gamma n_i + \delta(r_i+w_i)$ is an eigenvalue of $P_i=U_i +  \delta w_i I_{n_i}$ and  
$$spec(U(G)) = \bigg(\bigcup_{i=1}^k \big( spec(P_i)\backslash p_i \big) \bigg)\cup spec(\widetilde{U'}(G))$$
where $\widetilde{U'}(G)= \begin{bmatrix}
	p_1 & \sqrt{n_1n_2}\widehat{\rho}_{12}  & \cdots & \sqrt{n_1n_k}\widehat{\rho}_{1k}  \\
	\sqrt{n_2n_1}\widehat{\rho}_{21}  & p_2 & \cdots & \sqrt{n_2n_k}\widehat{\rho}_{2k} \\
	\vdots & \vdots & \ddots & \vdots  \\
	\sqrt{n_kn_1}\widehat{\rho}_{k1} & \sqrt{n_kn_2}\widehat{\rho}_{k2} & \cdots &p_k
	\end{bmatrix}.  $
\end{cor}
\begin{proof}
Clearly $\bold 1_{n_i}$  is an eigenvector of $P_i$ corresponding to the eigenvalue $p_i=\alpha r_i + \beta + \gamma n_i + \delta(r_i+w_i).$ 
Now, by Lemma \ref{evmain}, we have $\Gamma_i = \dfrac{n_i}{\lambda - p_i}$ and so  by Theorem \ref{chUofG}, we get 

$$\phi(U(G)) = \dfrac{\phi_{1}}{\lambda - p_1} \dfrac{\phi_{2}}{\lambda - p_2} \cdots \dfrac{\phi_{k}}{\lambda - p_k} n_1 n_2 \cdots n_{k} \det(\widetilde{U}(G)).  $$
Distributing $n_i$ inside the determinant of $\widetilde{U}(G)$, as $\sqrt{n_i}$ into the $i^{th}$ row and $\sqrt{n_i}$ into the $i^{th}$ column, we can write 
\begin{align*}
n_1 n_2 \cdots n_{k} \det(\widetilde{U}(G))&= \det \begin{bmatrix}
	(\lambda - p_1) & -\sqrt{n_1n_2}\widehat{\rho}_{12}  & \cdots & -\sqrt{n_1n_k}\widehat{\rho}_{1k}  \\
	-\sqrt{n_2n_1}\widehat{\rho}_{21}  & (\lambda - p_2) & \cdots & -\sqrt{n_2n_k}\widehat{\rho}_{2k} \\
	\vdots & \vdots & \ddots & \vdots  \\
	-\sqrt{n_kn_1}\widehat{\rho}_{k1} & -\sqrt{n_kn_2}\widehat{\rho}_{k2} & \cdots &(\lambda - p_k)
	\end{bmatrix}\\
	&=\det (\lambda I_n - \widetilde{U'}(G)).
\end{align*}
Now the proof follows.
\end{proof}

\begin{cor} \label{specUofG2}
Consider the notations defined in Theorem \ref{chUofG}. Suppose $\alpha+\delta=0$. Then $p_i=\beta + \gamma n_i + \delta w_i$ is an eigenvalue of $P_i=U_i +  \delta w_i I_{n_i}$ and  
$$spec(U(G)) = \bigg(\bigcup_{i=1}^k \big( spec(P_i)\backslash p_i \big) \bigg)\cup spec(\widetilde{U'}(G))$$
where $\widetilde{U'}(G)= \begin{bmatrix}
	p_1 & \sqrt{n_1n_2}\widehat{\rho}_{12}  & \cdots & \sqrt{n_1n_k}\widehat{\rho}_{1k}  \\
	\sqrt{n_2n_1}\widehat{\rho}_{21}  & p_2 & \cdots & \sqrt{n_2n_k}\widehat{\rho}_{2k} \\
	\vdots & \vdots & \ddots & \vdots  \\
	\sqrt{n_kn_1}\widehat{\rho}_{k1} & \sqrt{n_kn_2}\widehat{\rho}_{k2} & \cdots &p_k
	\end{bmatrix}.  $
\end{cor}
\begin{proof}
It is easy to see that $$(\alpha A_i + \beta I_{n_i} + \gamma J_{n_i} + \delta D_i)\bold 1_{n_i}=(\alpha+\delta)\begin{bmatrix}
deg_{G_i}(v_1^{(i)})\\ deg_{G_i}(v_2^{(i)}) \\ \vdots \\ deg_{G_i}(v_{n_i}^{(i)})
\end{bmatrix} + (\beta + \gamma n_i)\bold 1_{n_i}.$$ So $\bold 1_{n_i}$  is an eigenvector of $P_i$ corresponding to the eigenvalue $p_i=\beta + \gamma n_i + \delta w_i.$ 

Then by the same argument as in the previous corollary, the proof follows.
\end{proof}

\begin{rem}
For any graph $G$, the Laplacian matrix $L(G)$, is obtained from the universal adjacency matrix $U(G)$, by taking $(\alpha, \beta, \gamma, \delta)=(-1,0,0,1)$. So, we can find the Laplacian spectra of $H$-join of any graphs from Corollary \ref{specUofG2}.
\end{rem}

Let $H$ be a graph on $k$ vertices and $G'$ be any graph. We recall that the Lexicographic product of graphs $H$ and $G'$, denoted by $H[G']$, is obtained as the $H$-join of graphs in $\mathcal{F}=\{G_1,G_2, \dots, G_k\},$ where $G_i=G'$ for $1 \le i \le k$. In \cite{wng}, the authors obtained the characteristic polynomial of $H[G']$\cite[Theorem 2.4]{wng} and investigated the spectrum in various cases. Now we generalize \cite[Theorem 2.4]{wng} by obtaining the characteristic polynomial of the universal adjacency matrix of $H[G']$ when $\delta=0$. 

\begin{thm} \label{lexthm}
Let $H$ be a graph on $k$ vertices and $G'$ be a graph on $n'$ vertices. Consider the graph $G=H[G']$, the lexicographic product of $H$ and $G'$. Suppose $\spec(H)=\{\lambda_1,\lambda_2, \dots, \lambda_k \}.$ Then the characteristic polynomial of universal adjacency matrix of $U(G)$ when $\delta=0$, is 
$$\phi_{U(G)}(\lambda) = \phi^k(\lambda)\Bigg( \displaystyle \Pi_{i=1}^k (1-\lambda_i \Gamma(\lambda)) \Bigg),$$
where $\phi(\lambda)$ be the characteristic polynomial of $U(G')$ and $\Gamma(\lambda)=\Gamma_{U(G')}(\bold 1_{n'};\lambda)$ when $\delta=0$.
\end{thm}
\begin{proof}
By Theorem \ref{chUofG}, 
\begin{align*}
\phi_{U(G)}(\lambda)&= \phi^k(\lambda)\Gamma^k(\lambda)\det(\frac{1}{\Gamma(\lambda)}I_k-A(H))\\
&=\phi^k(\lambda)\Gamma^k(\lambda)\Bigg(\Pi_{i=1}^k (\frac{1}{\Gamma(\lambda)}-\lambda_i) \Bigg)\\
&=\phi^k(\lambda)\Bigg(\Pi_{i=1}^k (1- \lambda_i \Gamma(\lambda)) \Bigg).\\
\end{align*}
\end{proof}	

\subsection{The generalized characteristic polynomial of the $H$-join of graphs}\label{subsecgench}
The generalized characteristic polynomial of a graph $G$ is introduced in \cite{cvet}, as the bivariate polynomial defined by  $\phi_G(\lambda,t) = \det(\lambda I - (A(G) - t D(G)))$ where $A(G)$ and $D(G)$ are the adjacency and the degree matrix associated to the graph $G$. As mentioned earlier, in \cite[Theorem 3.1]{chen19} the authors obtained a generalization of Fiedler’s lemma, for the matrices with fixed row sum and as an application, they obtained the generalized characteristic polynomial of $H$-join of regular graphs.  In the following theorem, we obtain the generalized characteristic polynomial of $H$-join of any graphs.  

\begin{thm}\label{gchUofG}
Let $H$ be any graph and  $\mathcal{F}=\{G_1,G_2, \dots, G_k\}$ be a family of any $k$ graphs. Consider the graph $G=\displaystyle \bigvee_{H}\mathcal{F}$. Let $M(G)=A(G) - t D(G)$ and $M_i=A_i-t D_i$ for $1 \le i \le k$. Let $\phi_{i}$ be the characteristic polynomial of $M_i$ and $\Gamma_{i}=\Gamma_{M_i}(\textbf{1}_{n_i};\lambda)$. Then

i) The generalized characteristic polynomial of the graph $G$ is 
$$\phi_{M(G)}(\lambda) = \displaystyle \Pi_{i=1}^k \phi_i(\lambda+t w_i) \Gamma_i(\lambda+t w_i) \det(\widetilde{M}(G))$$

	where   $ \widetilde{M}(G) =   \begin{bmatrix}
	\frac{1}{\Gamma_1(\lambda+t w_1)} & -\rho_{12}   & \cdots & -\rho_{1k}  \\
	-\rho_{21}   & \frac{1}{\Gamma_2(\lambda+t w_2)} & \cdots & -\rho_{2k}  \\
	\vdots & \vdots & \ddots & \vdots  \\
	-\rho_{k1} & -\rho_{k2}  & \cdots &\frac{1}{\Gamma_{k}(\lambda+t w_3)}
	\end{bmatrix}$
	
ii)Analogous to the Equations \eqref{defnfg} and \eqref{maineqn1}, we define $f_i, g_i$, and $\Phi(\lambda)$  corresponding to the main eigenvalues of $M_i$ for $1 \le i \le k$. Then the spectrum of $M(G)$ is given as below.  

\begin{itemize}
\item For every eigenvalue $\mu$ of $M_i$ with multiplicity $m(\mu)$, which is not a main eigenvalue,   $\mu-t w_i$  is an eigenvalue of $M(G)$ with multiplicity $m(\mu)$.
\item For every main eigenvalue $\mu$ of $M_i$ with multiplicity $m(\mu)$, $\mu-t w_i$  is an eigenvalue of $M(G)$ with multiplicity $m(\mu)-1$.
\item Remaining eigenvalues are the roots of the polynomial $\Phi(\lambda)$.
\end{itemize}
\end{thm}
\begin{proof}
The proof is direct from Theorem \ref{chUofG}, by taking $(\alpha, \beta, \gamma, \delta)=(1,0,0,-t)$ in the universal adjacency matrix $U(G).$
\end{proof}

\begin{cor} \cite{chen19} \label{specGAofG}
Suppose the graph $G_i$ is $r_i$-regular for each $1 \le i \le k$. Then $p_i= r_i-t(r_i+w_i)$ is an eigenvalue of $P_i=M_i -t w_i I_{n_i}$ and  
$$spec(M(G)) = \bigg(\bigcup_{i=1}^k \big( spec(P_i)\backslash p_i \big) \bigg)\cup spec(\widetilde{M'}(G))$$
where $\widetilde{U'}(G)= \begin{bmatrix}
	p_1 & \sqrt{n_1n_2}{\rho}_{12}  & \cdots & \sqrt{n_1n_k}{\rho}_{1k}  \\
	\sqrt{n_2n_1}{\rho}_{21}  & p_2 & \cdots & \sqrt{n_2n_k}{\rho}_{2k} \\
	\vdots & \vdots & \ddots & \vdots  \\
	\sqrt{n_kn_1}{\rho}_{k1} & \sqrt{n_kn_2}{\rho}_{k2} & \cdots &p_k.
	\end{bmatrix}.  $
\end{cor}
\begin{proof}
The proof is obtained from Theorem \ref{gchUofG} by the similar arguments used in the Corollary \ref{specUofG} of Theorem \ref{chUofG}.
\end{proof}
\begin{rem}
In \cite{chen19}, from the generalized characteristic polynomial, the spectra of the signless Laplacian, the normalized laplacian are deduced for $H$-join of regular graphs.  Similarly, we too can deduce them for $H$-join of any graphs from Theorem \ref{gchUofG}. Also, we can deduce the Seidel spectra of $H$-join of any graphs by taking $(\alpha, \beta, \gamma, \delta)=(-2,-1,1,0)$ in the universal adjacency matrix $U(G)$ in Theorem \ref{chUofG}.
\end{rem}

\section{Spectra of the $H$-generalized join of graphs}\label{subsechgenjoin}
In this section, we obtain the characteristic polynomial of $H$-generalized join of graphs  $\displaystyle \bigvee_{H,\mathcal{S}}\mathcal{F}$ introduced in \cite{cdo13}. 
\begin{defn}
	Let $G$ be any graph. A vertex subset $S$ of a graph $G$ is said to be $(k,\tau)$-regular if $S$ induces a $k$-regular graph in $G$ and  every vertex outside of $S$ has $\tau$ neighbours in $S$. When $G$ is a regular graph, for convenience $S = V(G)$  is considered as $(k,0)$-regular. 
\end{defn}
\begin{defn}
	Let $G$ be any graph with vertex set $\{v_1,v_2,\dots,v_n\}$. For any subset $S\subset V(G)$, the characteristic vector of $S$, denoted by  $\chi_S$, is defined as the 0-1 vector such that $i^{th}$ place of $\chi_S$ is 1 if and only if the vertex $v_i \in S.$ 
\end{defn}
\begin{lem}\cite{cdo13}\label{mainlem}
	Let $G$ be a graph with a $(k,\tau)$-regular set $S$, where $\tau > 0$, and $\lambda \in \sigma(A(G))$. Then, $\lambda$ is not a main eigenvalue if and only if $\lambda = k - \tau$ or $\chi_S \in (\mathcal E_G(\lambda))^{\perp}$.
\end{lem}

	Fix a $(k,\tau)$-regular subset $S$ of $V(G)$. An eigenvalue  $\lambda \in \sigma(G)$ is said to be \textit{special eigenvalue} if  $\lambda \ne k -\tau $ and $\lambda$ is not a main eigenvalue. Then by Lemma \ref{mainlem}, if $\lambda$ is a special eigenvalue of $G$ then $\lambda$ is not a $\chi_s$-main eigenvalue. In \cite{cdo13} the authors obtained all eigenvalues of $\displaystyle \bigvee_{H,\mathcal{S}}\mathcal{F}$ when $G_i$ is regular and the subsets $S_i \in \mathcal{S}$ are such that $S_i=V(G_i)$ for $1 \le i \le k$, in which case  $\displaystyle \bigvee_{H,\mathcal{S}}\mathcal{F}$ coincides with the $H$-join of regular graphs $\displaystyle \bigvee_{H}\mathcal{F}$. In other cases, it is proved that every special eigenvalue corresponding to $(k_i,\tau_i)$-regular subset $S_i$ is an eigenvalue of $\displaystyle \bigvee_{H,\mathcal{S}}\mathcal{F}$ and thus the subset of eigenvalues is obtained for $\displaystyle \bigvee_{H,\mathcal{S}}\mathcal{F}$.  
	
	In the following theorem, we obtain the characteristic polynomial of $\displaystyle \bigvee_{H,\mathcal{S}}\mathcal{F}$ for any family of subsets $\mathcal{S}$ and obtain the complete set of eigenvalues.
\begin{thm}\label{thmghjoin}
	Consider a graph $H$ of order $k$ and a family of graphs $\mathcal F = \{G_1,\dots,G_k\}$. Consider also a family of vertex subsets $\mathcal S = \{S_1, \dots, S_k\}$, such that	$S_i \in V(G_i)$ for $1 \le i \le k$. Let $G =\displaystyle \bigvee_{H,S} \mathcal F$. Let $n_i$ and $A_i$ be the number of vertices and the adjacency matrix of the graph $G_i$ respectively for $1\le i \le k$. For $1 \le i,j \le k$, let $\rho_{ij}$ be the scalars defined by $\rho_{ij} = 1 \text{ if}$ $ij \in E(H)$ and  $0 \text{ otherwise}$. Then we have the following. \\
	i) The characteristic polynomial of $G$ is
$$\phi_{G}(\lambda) = \displaystyle \Pi_{i=1}^k \phi_i(\lambda) \Gamma_i(\lambda) \det(\widetilde{A}(G))$$
\begin{equation}\label{gjoinAdofG}
	\text{where}\,\,    \widetilde{A}(G) =   \begin{bmatrix}
	\frac{1}{\Gamma_1} & -\rho_{12}  & \cdots & -\rho_{1k}   \\
	-\rho_{21}  & \frac{1}{\Gamma_2} & \cdots & -\rho_{2k} \\
	\vdots & \vdots & \ddots & \vdots  \\
	-\rho_{k1} & -\rho_{k2}  & \cdots &\frac{1}{\Gamma_{k}}.
	\end{bmatrix}
\end{equation}		
where $\phi_i(\lambda)=\det(\lambda I_{n_i}-A(G_i))$ and $\Gamma_i(\lambda)=\Gamma_{A_i}(\chi_{S_{i}};\lambda)$

\noindent ii) Analogous to the Equations \eqref{defnfg} and \eqref{maineqn1}, we define $f_i, g_i$ and $\Phi(\lambda)$  corresponding to the $\chi_{S_{i}}$-main eigenvalues of $G_i$ for $1 \le i \le k$. Then the spectrum of $G$ is given as below. 
\begin{itemize}
\item Every eigenvalue $\mu$ of $A_i$ with multiplicity $m(\mu)$, which is not $\chi_{S_{i}}$-main eigenvalue,  is an eigenvalue of $G$ with multiplicity $m(\mu)$.
\item Every $\chi_{S_{i}}$-main eigenvalue $\mu$ of $A_i$ with multiplicity $m(\mu)$,  is an eigenvalue of $G$ with multiplicity $m(\mu)-1$.
\item Remaining eigenvalues are the roots of the polynomial $\Phi(\lambda).$
\end{itemize}
\end{thm}
\begin{proof}
By the definition of $\bigvee_{(H,S)} \mathcal F$, the adjacency matrix of $G$ is given as 
$$ A(G) = \begin{bmatrix}
A_1  & \rho_{12} \chi_{S_{1}}\chi_{S_{2}}^t & \cdots & \rho_{1k} \chi_{S_{1}}\chi_{S_{k}}^t \\
\rho_{21} \chi_{S_{2}}\chi_{S_{1}}^t & A_2 & \cdots & \rho_{2k} \chi_{S_{2}}\chi_{S_{k}}^t \\
\vdots & \vdots & \ddots & \vdots  \\
\rho_{k1} \chi_{S_{k}}\chi_{S_{1}}^t & \rho_{k2} \chi_{S_{k}}\chi_{S_{2}}^t & \cdots & A_k
\end{bmatrix}.$$
Then by direct application of Theorem \ref{mainthm} and Theorem \ref{mainthm2} on $A(G)$, the proofs of (i) and (ii) follow immediately.
\end{proof}
	
\section{Spectra of generalized corona of graphs}
In \cite[Theorem 3.1]{lsk}, the generalized corona product is defined as below and its characteristic polynomial is calculated. In this subsection, we deduce this result as a corollary of Theorem \ref{mainthm}. This is done by viewing the corona product as $H$-join of suitably chosen graphs. 

\begin{defn}
	Let $H'$ be a graph on $k$ vertices. Let $G_1,G_2,\dots,G_k$ be graphs of order $n_1,n_2,\dots,n_k$ respectively. The generalized corona product of $H'$ with $G_1,G_2,\dots,G_k$, denoted by $H' \tilde{\circ} \Lambda_{i=1}^n G_i$, is obtained by taking one copy of graphs $H',G_1,G_2,\dots,G_k$, and joining the $i^{th} $ vertex of $H'$ to every vertex of $G_i$. 
\end{defn}
When $G_i=G'$ for all $i$, the graph $H' \tilde{\circ} \Lambda_{i=1}^n G_i$ is called simply corona of $H'$ and $G'$, denoted by $H' \circ G'$.
\begin{thm}\label{scorona}
	Let $H'$ be a graph with vertex set $V(H')=\{v_1,v_2,\dots,v_k\}$. Let $G_1,G_2,\dots,G_k$ be any graphs. $\rho_{ij} = 1 \text{ if}$ $v_iv_j \in E(H)$ and  $0 \text{ otherwise}$. The characteristic polynomial of the generalized corona product $G = H' \tilde{\circ} \Lambda_{i=1}^k G_i$ is given by
	$$\phi_G(\lambda)= \Pi_{i=1}^k \phi_{G_i}(\lambda) \det(\widetilde A(H')) $$ where 
	$\widetilde A(H')=\begin{bmatrix}
	\lambda-\Gamma_{G_1}(\lambda) & -\rho_{12}  & \cdots & -\rho_{1,k}  \\
	-\rho_{21}  & \lambda-\Gamma_{G_2}(\lambda) & \cdots & -\rho_{2,k} \\
	\vdots & \vdots & \cdots & \vdots  \\
	-\rho_{k1} & -\rho_{k2} & \cdots & \lambda-\Gamma_{G_k}(\lambda)
	\end{bmatrix}$
\end{thm}
\begin{proof}
Let $H=H' \circ K_1.$ Let $v_{k+i}$ be the new vertex in $H$ attached with the vertex $v_i$ in the copy of $H'$, for $1 \le i \le k$. Let $\mathcal{F}=\{K_1, K_1, \dots, K_1, G_1, G_2, \dots, G_k\}$. Then we get the following visualization of generalized corona as $H$-join of graphs in $\mathcal{F}$. $$(H' \tilde{\circ} \Lambda_{i=1}^n G_i)=(\displaystyle \bigvee_{H}\mathcal{F}) $$
That is, each $v_i$ is replaced by $K_1$ and $v_{k+i}$ is replaced by $G_i$ in $H$, to form the $H$-join. 

Now $A(H)= \begin{bmatrix}
	A(H') & I_k \\ I_k & 0_k
	\end{bmatrix}$. Since $\phi_{K_1}(\lambda)=\lambda$ and $\Gamma_{K_1}(\lambda)=\dfrac{1}{\lambda}$, by letting $\alpha=1$, $\beta=\gamma=\delta=0$ in Theorem \ref{chUofG}, we get 
		$$\phi_G(\lambda) =\Bigg(\Pi_{i=1}^k (\phi_{K_1}(\lambda) \phi_{G_i}(\lambda) \Gamma_{K_1}(\lambda) \Gamma_{G_i}(\lambda))\Bigg) \det(\widetilde A(H))$$ which implies
	\begin{equation} \label{pfcorona}
\phi_G(\lambda)=\Bigg(\Pi_{i=1}^k (\phi_{G_i}(\lambda)\Gamma_{G_i}(\lambda))\Bigg)\det(\widetilde A(H))
	\end{equation}
where $\widetilde A(H)=\begin{bmatrix}
	\lambda I_k-A(H') & -I_k \\ -I_k & diag(\dfrac{1}{\Gamma_{G_1}(\lambda)},\dfrac{1}{\Gamma_{G_2}(\lambda)},\dots,\dfrac{1}{\Gamma_{G_k}(\lambda)})
	\end{bmatrix}$.\\
Now by Lemma \ref{lemd}, $\det(\widetilde H)$ is given as\\
 $\det\begin{bmatrix}
\frac{1}{\Gamma_{G_1}(\lambda)} & 0 & \cdots & 0\\
0 & \frac{1}{\Gamma_{G_2}(\lambda)} & \cdots & 0\\
\vdots& \vdots & \ddots & \vdots\\
0 & 0 & \cdots & \frac{1}{\Gamma_{G_k}(\lambda)}
\end{bmatrix}  \det \Bigg(\lambda I_k-A(H')-
\begin{bmatrix}
\Gamma_{G_1}(\lambda) & 0 & \cdots & 0\\
0 & \Gamma_{G_2}(\lambda) & \cdots & 0\\
\vdots& \vdots & \ddots & \vdots\\
0 & 0 & \cdots & \Gamma_{G_k}(\lambda))
\end{bmatrix} \Bigg)  $ 

$=\Bigg(\Pi_{i=1}^k \dfrac{1}{\Gamma_{G_i}(\lambda)}\Bigg) \det\bigg(\begin{bmatrix}
\lambda-\Gamma_{G_1}(\lambda) & 0 & \cdots & 0\\
0 & \lambda-\Gamma_{G_2}(\lambda) & \cdots & 0\\
\vdots& \vdots & \ddots & \vdots\\
0 & 0 & \cdots & \lambda-\Gamma_{G_k}(\lambda))
\end{bmatrix}-A(H')\bigg)$

Now by substituting $ \det(\widetilde H)$ in Equation \eqref{pfcorona} we get the required result.
\end{proof}
\begin{rem} Similarly, we can get the other variants of the spectra of the generalized corona of any graphs. The same work can be done on other variants of corona also by suitable choice of $H$.    
\end{rem}
 
\section{Examples}\label{examples}
To illustrate our results, we compute the characteristic polynomials of two particular examples on $H$-join of graphs and $H$-generalized join of graphs constrained by vertex subsets. Similarly, we can apply our other results also.

\begin{example}\label{ex1}
	Consider the graphs $H=P_3, G_1=P_3,G_2=K_{1,3}$ and $G_3=K_2$ as follows. 
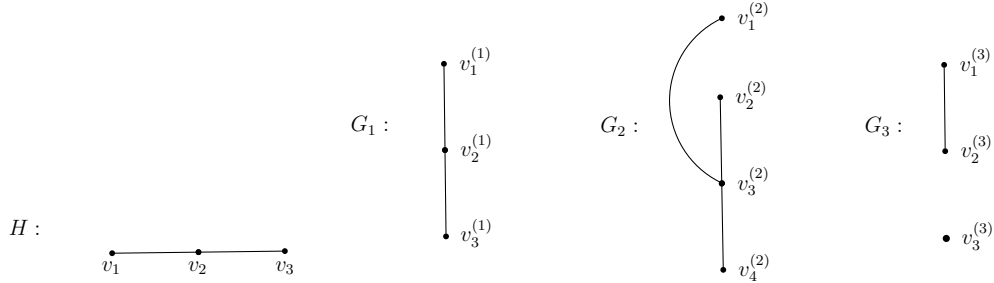
\begin{center}
	\scalebox{0.7} 
	{
		\begin{pspicture}(0,-0.59375)(5.64375,0.59375)
		\psline[linewidth=0.02cm,dotsize=0.07055555cm 2.0]{*-*}(1.7846875,-0.08875)(3.4246874,-0.06875)
		\psline[linewidth=0.02cm,dotsize=0.07055555cm 2.0]{*-*}(3.4246874,-0.06875)(5.0646877,-0.04875)
		\usefont{T1}{ptm}{m}{n}
		\rput(1.7692188,-0.35875){$v_1$}
		\usefont{T1}{ptm}{m}{n}
		\rput(3.4092188,-0.35875){$v_2$}
		\usefont{T1}{ptm}{m}{n}
		\rput(5.0692186,-0.35875){$v_3$}
		\usefont{T1}{ptm}{m}{n}
		\rput(0.12671874,0.40125){$H:$}
		\end{pspicture} 
	}
	\scalebox{0.7} 
	{
		\begin{pspicture}(0,-2.6092188)(13.309063,2.6092188)
		\psline[linewidth=0.02cm,dotsize=0.07055555cm 2.0]{*-*}(1.9699908,1.4957812)(1.99,-0.14421874)
		\psline[linewidth=0.02cm,dotsize=0.07055555cm 2.0]{*-*}(1.99,-0.14421874)(2.010009,-1.7842187)
		\usefont{T1}{ptm}{m}{n}
		\rput(0.56453127,0.32578126){$G_1:$}
		\rput{90.43272}(8.840372,-7.1823015){\psarc[linewidth=0.02,dotsize=0.07055555cm 2.0]{*-*}(7.9843173,0.7957781){1.73}{24.734303}{154.47658}}
		\psline[linewidth=0.02cm,dotsize=0.07055555cm 2.0]{*-*}(7.2170687,0.86002195)(7.246196,-0.7798413)
		\psline[linewidth=0.02cm,dotsize=0.07055555cm 2.0]{*-*}(7.246196,-0.7798413)(7.275323,-2.4197047)
		\psline[linewidth=0.02cm,dotsize=0.07055555cm 2.0]{*-*}(11.469991,1.4757811)(11.49,-0.16421875)
		\psdots[dotsize=0.136](11.51,-1.8242188)
		\usefont{T1}{ptm}{m}{n}
		\rput(2.5745313,1.5457813){$v_1^{(1)}$}
		\usefont{T1}{ptm}{m}{n}
		\rput(2.5945313,-0.09421875){$v_2^{(1)}$}
		\usefont{T1}{ptm}{m}{n}
		\rput(2.5945313,-1.7142187){$v_3^{(1)}$}
		\usefont{T1}{ptm}{m}{n}
		\rput(7.8545313,2.4057813){$v_1^{(2)}$}
		\usefont{T1}{ptm}{m}{n}
		\rput(5.304531,0.30578125){$G_2:$}
		\usefont{T1}{ptm}{m}{n}
		\rput(10.324532,0.30578125){$G_3:$}
		\usefont{T1}{ptm}{m}{n}
		\rput(7.8145313,0.90578127){$v_2^{(2)}$}
		\usefont{T1}{ptm}{m}{n}
		\rput(7.8545313,-0.7342188){$v_3^{(2)}$}
		\usefont{T1}{ptm}{m}{n}
		\rput(7.8745313,-2.3742187){$v_4^{(2)}$}
		\usefont{T1}{ptm}{m}{n}
		\rput(12.074532,1.5257813){$v_1^{(3)}$}
		\usefont{T1}{ptm}{m}{n}
		\rput(12.094531,-0.11421875){$v_2^{(3)}$}
		\usefont{T1}{ptm}{m}{n}
		\rput(12.1145315,-1.7742188){$v_3^{(3)}$}
		\end{pspicture} 
	}
	
\end{center}
Let $\mathcal{F}=\{G_1,G_2,G_3 \}$. Then the $H$-join graph $G=\displaystyle \bigvee_{H}\mathcal{F}$ is given as 
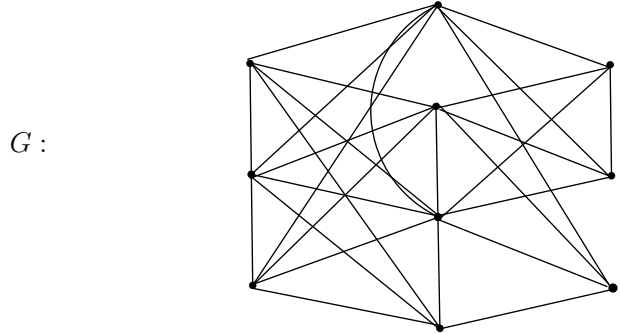
\begin{center}
	\scalebox{0.9} 
	{
		\begin{pspicture}(0,-2.455)(8.867687,2.455)
		\psline[linewidth=0.02cm,fillcolor=black,dotsize=0.07055555cm 2.0]{*-*}(3.4196784,1.5299999)(3.4396875,-0.11)
		\psline[linewidth=0.02cm,fillcolor=black,dotsize=0.07055555cm 2.0]{*-*}(3.4396875,-0.11)(3.4596968,-1.7499999)
		\usefont{T1}{ppl}{m}{n}
		\rput(0.14921875,0.36){$G:$}
		\rput{90.43272}(7.8163457,-6.097542){\psarc[linewidth=0.02,dotsize=0.07055555cm 2.0]{*-*}(6.934005,0.8299969){1.73}{24.734303}{154.47658}}
		\psline[linewidth=0.02cm,fillcolor=black,dotsize=0.07055555cm 2.0]{*-*}(6.166756,0.89424074)(6.1958833,-0.7456226)
		\psline[linewidth=0.02cm,fillcolor=black,dotsize=0.07055555cm 2.0]{*-*}(6.1958833,-0.7456226)(6.2250104,-2.385486)
		\psline[linewidth=0.02cm,fillcolor=black,dotsize=0.07055555cm 2.0]{*-*}(8.739678,1.5099999)(8.759687,-0.13)
		\psdots[dotsize=0.136](8.779688,-1.79)
		\psline[linewidth=0.02cm](3.3796875,1.59)(6.1996875,2.41)
		\psline[linewidth=0.02cm](3.3796875,1.55)(6.1196876,0.89)
		\psline[linewidth=0.02cm](3.4396875,1.53)(6.2196875,-2.37)
		\psline[linewidth=0.02cm](3.4396875,-0.11)(6.1996875,2.43)
		\psline[linewidth=0.02cm](3.4196875,-0.17)(6.1996875,0.91)
		\psline[linewidth=0.02cm](3.4396875,-0.11)(6.2196875,-0.73)
		\psline[linewidth=0.02cm](3.4196875,-0.13)(6.2596874,-2.43)
		\psline[linewidth=0.02cm](3.4596875,-1.75)(6.2196875,2.43)
		\psline[linewidth=0.02cm](3.4596875,-1.75)(6.1596875,0.89)
		\psline[linewidth=0.02cm](3.4796875,-1.73)(6.2396874,-0.73)
		\psline[linewidth=0.02cm](3.4396875,-1.79)(6.2596874,-2.35)
		\psline[linewidth=0.02cm](6.2396874,2.37)(8.739688,1.47)
		\psline[linewidth=0.02cm](6.1796875,2.37)(8.719687,-0.13)
		\psline[linewidth=0.02cm](6.2196875,2.33)(8.739688,-1.77)
		\psline[linewidth=0.02cm](6.1796875,0.87)(8.719687,1.47)
		\psline[linewidth=0.02cm](6.2196875,0.85)(8.719687,-0.13)
		\psline[linewidth=0.02cm](6.1796875,0.95)(8.719687,-1.73)
		\psline[linewidth=0.02cm](6.2396874,-0.75)(8.779688,1.51)
		\psline[linewidth=0.02cm](6.2396874,-0.71)(8.739688,-0.15)
		\psline[linewidth=0.02cm](6.1796875,-0.77)(8.699688,-1.77)
		\psline[linewidth=0.02cm](6.2396874,-2.39)(8.779688,-1.81)
		\psline[linewidth=0.02cm](3.4396875,1.53)(6.2196875,-0.73)
		\end{pspicture} 
	}
\end{center}

We see that 
	$\phi_1(\lambda)=\lambda^3-2\lambda $, $\phi_2(\lambda)=\lambda^4-3\lambda^2$, $\phi_3(\lambda)=\lambda^3-\lambda $, $\Gamma_1(\lambda)=\dfrac{3\lambda+4}{\lambda^2-2}$,  $\Gamma_2(\lambda)=\dfrac{4\lambda+6}{\lambda^2-3}$ and  $\Gamma_3(\lambda)=\dfrac{3\lambda-1}{\lambda^2-\lambda}$. The characteristic polynomial of $G$ is $\lambda^3 (\lambda^3+4\lambda^2-\lambda-6)(\lambda^3-5\lambda^2-8\lambda+2)(\lambda+1)$ which is equal to $$\phi_1(\lambda)\phi_2(\lambda)\phi_3(\lambda)\Gamma_1(\lambda)\Gamma_2(\lambda)\Gamma_3(\lambda) \det\begin{bmatrix}
	\dfrac{1}{\Gamma_1(\lambda)}&-1&0\\
	-1&\dfrac{1}{\Gamma_2(\lambda)}&-1\\
	0& -1 &\dfrac{1}{\Gamma_3(\lambda)}
	\end{bmatrix}.$$
 
\end{example}


\begin{example}\label{ex2}
	Consider $H$ and $\mathcal{F}$ as in Example \ref{ex1}. Let $S_1=\{v_1^{(1)},v_2^{(1)}\}$, $S_2=\{v_1^{(2)},v_2^{(2)},v_4^{(2)}\}$ and $S_3=\{v_2^{(3)},v_3^{(3)}\}$. Then the $H$-generalized join graph $G=\displaystyle \bigvee_{H,S}\mathcal{F}$ is given as 
	\begin{center}
		\scalebox{0.9} 
		
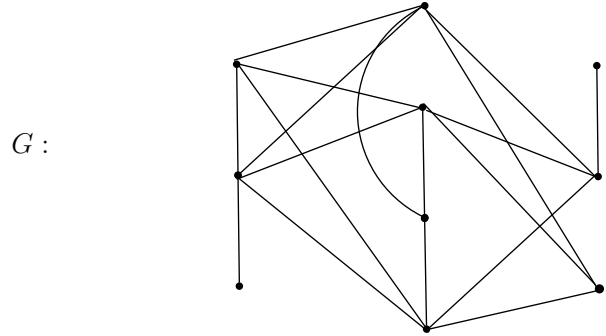
{
			\begin{pspicture}(0,-2.455)(8.858,2.455)
			\psline[linewidth=0.02cm,fillcolor=black,dotsize=0.07055555cm 2.0]{*-*}(3.4099908,1.5299999)(3.43,-0.11)
			\psline[linewidth=0.02cm,fillcolor=black,dotsize=0.07055555cm 2.0]{*-*}(3.43,-0.11)(3.450009,-1.7499999)
			\usefont{T1}{ppl}{m}{n}
			\rput(0.36453125,0.36){$G:$}
			\rput{90.43272}(7.806585,-6.087855){\psarc[linewidth=0.02,dotsize=0.07055555cm 2.0]{*-*}(6.9243174,0.8299969){1.73}{24.734303}{154.47658}}
			\psline[linewidth=0.02cm,fillcolor=black,dotsize=0.07055555cm 2.0]{*-*}(6.1570687,0.89424074)(6.186196,-0.7456226)
			\psline[linewidth=0.02cm,fillcolor=black,dotsize=0.07055555cm 2.0]{*-*}(6.186196,-0.7456226)(6.215323,-2.385486)
			\psline[linewidth=0.02cm,fillcolor=black,dotsize=0.07055555cm 2.0]{*-*}(8.729991,1.5099999)(8.75,-0.13)
			\psdots[dotsize=0.136](8.77,-1.79)
			\psline[linewidth=0.02cm](3.37,1.59)(6.19,2.41)
			\psline[linewidth=0.02cm](3.37,1.55)(6.11,0.89)
			\psline[linewidth=0.02cm](3.43,1.53)(6.21,-2.37)
			\psline[linewidth=0.02cm](3.43,-0.11)(6.19,2.43)
			\psline[linewidth=0.02cm](3.41,-0.17)(6.19,0.91)
			\psline[linewidth=0.02cm](3.41,-0.13)(6.25,-2.43)
			\psline[linewidth=0.02cm](6.17,2.37)(8.71,-0.13)
			\psline[linewidth=0.02cm](6.21,2.33)(8.73,-1.77)
			\psline[linewidth=0.02cm](6.21,0.85)(8.71,-0.13)
			\psline[linewidth=0.02cm](6.17,0.95)(8.71,-1.73)
			\psline[linewidth=0.02cm](6.23,-2.39)(8.77,-1.81)
			\psline[linewidth=0.02cm](6.25,-2.37)(8.73,-0.09)
			\end{pspicture} 
		}
		
	\end{center}
	Here, $\phi_1(\lambda)=\lambda^3-2\lambda $, $\phi_2(\lambda)=\lambda^4-3\lambda^2$ and $\phi_3(\lambda)=\lambda^3-\lambda $. Based on the choices of $S_1,S_2$ and $S_3$ we get $\Gamma_1(\chi_{S_1};\lambda)=\dfrac{2\lambda^2+2\lambda-1}{\lambda^3-2\lambda}$,  $\Gamma_2(\chi_{S_2};\lambda)=\dfrac{3\lambda}{\lambda^2-3}$ and  $\Gamma_3(\chi_{S_3};\lambda)=\dfrac{2\lambda^2-1}{\lambda^3-\lambda}$. The characteristic polynomial of $G$ is $\lambda^4(\lambda^6-18\lambda^4-6\lambda^3+35\lambda^2+6\lambda-15)$ which is equal to $$\phi_1(\lambda)\phi_2(\lambda)\phi_3(\lambda)\Gamma_1(\chi_{S_1};\lambda)\Gamma_2(\chi_{S_2};\lambda)\Gamma_3(\chi_{S_3};\lambda) \det\begin{bmatrix}
	\dfrac{1}{\Gamma_1(\chi_{S_1})}&-1&0\\
	-1&\dfrac{1}{\Gamma_2(\chi_{S_2})}&-1\\
	0& -1 &\dfrac{1}{\Gamma_3(\chi_{S_3})}
	\end{bmatrix}.$$ 
\end{example}

\end{document}